\documentclass[12pt]{extarticle}
\usepackage{extsizes}
\usepackage[reqno]{amsmath} 
\usepackage{amsmath}
\usepackage{amssymb}
\usepackage{amsthm}
\usepackage{amscd}
\usepackage{amsfonts}
\usepackage{graphicx}
\usepackage{dsfont}
\usepackage{fancyhdr}
\usepackage{enumerate}
\usepackage{youngtab}
\usepackage[utf8]{inputenc}
\usepackage{comment}
\usepackage{mathrsfs}

\usepackage{subfigure}
\usepackage[margin=1.5in]{geometry}
\usepackage{graphicx}
\usepackage{algorithm}
\usepackage{algpseudocode}
\usepackage{color}
\usepackage{hyperref}
\usepackage{notation}

\theoremstyle{theorem}

\newtheorem{corollary}{Corollary}

\newtheorem{lemma}[corollary]{Lemma}

\newtheorem{lemma*}[lem6]{Lemma}

\newtheorem{theorem}[corollary]{Theorem}

\begin{document}

\AtEndDocument{%
  \par
  \medskip
  \begin{tabular}{@{}l@{}}%
    \textsc{Emanuel Juliano} \\
    \textsc{Dept. of Computer Science} \\ 
    \textsc{Universidade Federal de Minas Gerais, Brazil} \\
    \textit{E-mail address}: \texttt{emanuelsilva@dcc.ufmg.br}
    
  \end{tabular}}

\title{Forbidden subdivision in integral trees}
\author{Emanuel Juliano\footnote{emanuelsilva@dcc.ufmg.br --- remaining affiliations in the end of the manuscript.}}
\date{\today}
\maketitle
\vspace{-0.8cm}

\begin{abstract} 
We show that if all the eigenvalues of a tree are integers, then it does not contain a subdivided edge with 7 vertices.
\end{abstract}

\begin{center}
\textbf{Keywords}
Integral Trees ; Subdivision
\end{center}

\section{Introduction}\label{intro}

An integral tree is a tree for which the eigenvalues of its adjacency matrix are all integers \cite{harary1974graphs}. Many constructions of integral trees have been found lately \cite{csikvari2010integral}. However, all known constructions and examples of integral trees do not contain a subdivided edge with more than 3 vertices, indicating that this structure might forbid the integrality of the spectrum.

Let $T$ be a tree and $P_{a, b}$ denote the set of vertices in the path between $a$ and $b$ in $T$. We say $P_{a, b}$ is a subdivided edge if all inner vertices of the path have degree $2$. Coutinho et al. \cite{coutinho2023spectrum} show that no integral tree contains a subdivided edge with $8$ vertices; we reduce this size to $7$.

\begin{theorem}\label{thm:1}
If a tree $T$ contains a subdivided edge with $7$ vertices, then the tree has at least one eigenvalue that is not an integer.
\end{theorem}

We assume our tree $T$ has the following format:
\begin{figure}[H]
\begin{center}
	\includegraphics[width=10cm]{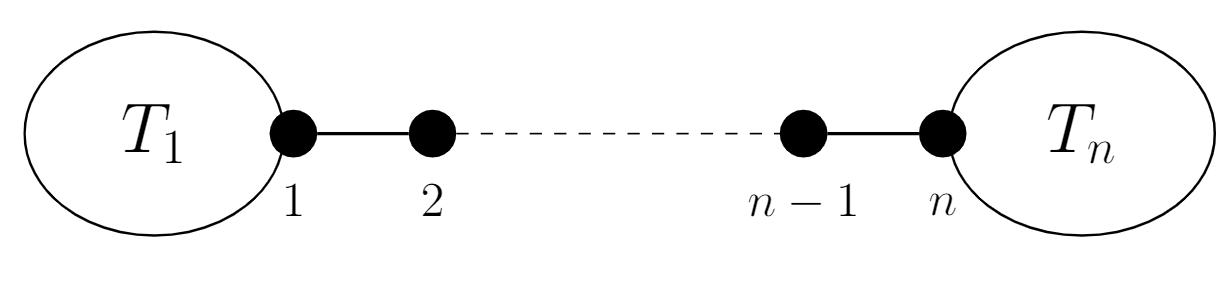}
	\caption{A tree that contains two vertices separated by a subdivided edge with $n$ vertices.} \label{fig:treepath}
\end{center}
\end{figure}

\section{Proof of Theorem}

Similarly to the argument used in \cite{coutinho2023spectrum}, we make use of an alternative interpretation of the algorithm developed by Jacobs and Trevisan \cite{TrevisanJacobsOriginal} to locate eigenvalues of trees.

Assume $T$ is a rooted tree and make vertex $1$ the root. For any vertex $i$ of $T$, let $T(i)$ denote the subtree of vertex $i$, that is, the induced subgraph corresponding to vertex $i$ and all its descendants. We define the rational function on the variable $x$ by
 \begin{equation}\label{eq:algorithm}
 	d_i = x - \sum_{j \text{ child of }i} \frac{1}{d_j}.
 \end{equation}


In order to use $d_i$ to compute the distinct eigenvalues of $T$ within an interval, we employ the following lemma. Let $\phi^G$ denote the characteristic polynomial of the adjacency matrix $A(G)$ of a graph $G$.

\begin{lemma}
Let $T$ be a rooted tree,
    \begin{equation} \label{eq:alpha}
    d_i = \frac{\phi^{T(i)}}{\phi^{T(i) \setminus i}}.
    \end{equation}
\end{lemma}
\begin{proof}
The identity holds true if $|T(i)| = 1$. To demonstrate the general case, we apply induction and analyze the Leibniz determinant formula. Since we are computing the determinant of a tree, the only permutations that contribute are those consisting of disjoint transpositions corresponding to edges of $T(i)$. Let $\mathcal{S}$ be the set of such permutations. Then,
\begin{equation*}
\begin{split}
\phi^{T(i)} &= \det(xI - A(T(i))) \\
&= x \sum_{\substack{\sigma \in \mathcal{S} \\
\sigma \text{ fixes } i}} \text{sgn}(\sigma) \prod_{k \in T(i) \setminus i}   (xI - A(T(i)))_{k, \sigma_k} \\ 
& \quad + \sum_{j \text{ child of }i} \sum_{\substack{\sigma \in \mathcal{S} \\
(i j) \in \sigma }} \text{sgn}(\sigma) \prod_{k \in T(i) \setminus \{i, j\}}   (xI - A(T(i)))_{k, \sigma_k} \\
&= x \phi^{T(i) \setminus i} - \sum_{j \text{ child of }i} \phi^{T(i) \setminus \{i, j\}} \\
&= x \phi^{T(i) \setminus i} - \sum_{j \text{ child of }i} \frac{\phi^{T(i) \setminus i}}{d_j}.
\end{split}
\end{equation*}
\end{proof}

The above lemma implies that if $d_i(\theta) = 0$, then $\theta$ is an eigenvalue of $T(i)$. Therefore, to lower bound the number of eigenvalues of $T$ within an interval, it suffices to determine how many times $d_i$ becomes equal to 0.

We also use the following analytical properties of $d_i$.

\begin{lemma}\label{lem:di}
    The function $d_i$ is odd and its derivative $d_i'(\theta)$ is greater than or equal to $1$ for every $\theta$ that is not a pole of $d_i$.
\end{lemma}
\begin{proof}
    These assertions can be shown by induction on the child vertices. For a leaf, $d_i = x$ and satisfies the statement. We check that the function is odd:
\begin{equation*}
    d_i(-\theta) = -\theta - \sum_{j \text{ child of }i} \frac{1}{d_j(-\theta)} = -\theta - \sum_{j \text{ child of }i} \frac{1}{-d_j(\theta)} = -d_i(\theta),
\end{equation*}
and has derivative $\geq 1$:
\begin{equation*}
    d_i'(\theta) = 1 - \sum_{j \text{ child of }i} \left(\frac{1}{d_j(\theta)}\right)' = 1 + \sum_{j \text{ child of }i} d_j(\theta)^{-2} d_j'(\theta) \geq 1.
\end{equation*}
Where the last inequality follows from induction on $d_j'(\theta)$ and the fact that $d_j(\theta)^{-2} \geq 0$ for every $\theta$ that is not a pole of $d_i$. 
\end{proof}

Let $T$ be a tree as described in figure \ref{fig:treepath}, and assume $n=7$. We aim to demonstrate that there are too many distinct eigenvalues in the interval $(-2,2)$, such that at least one is not an integer. Let vertex $1$ be the root. We utilize the same observation as in \cite{coutinho2023spectrum} to initially bound the number of distinct eigenvalues:

\begin{itemize}    
    \item Make $\theta = 2$. If $i>1$, then
    	\[d_i(2) = 2-\frac{1}{d_{i+1}(2)},\] 
    So once there is a negative value or a zero on the path, all the remaining values going towards the root become positive or poles, as $2-1/x$ maps $[1,\infty]$ to itself. Therefore, the number of positive values or poles in the path is at least $n-2$ (at most one negative or one zero, and we cannot control what happens at vertex $1$).
    \item Make $\theta = -2$. Since the function is odd, there can be at most $2$ positive values or poles in the path.
\end{itemize}

\begin{proof}[proof of Theorem \ref{thm:1}]
    Our goal is to count how many times the root becomes equal to $0$ in the interval $(-2, 2)$, as this also determines the number of distinct eigenvalues in the interval.
    
    To achieve this, we examine the dynamics of the sign changes inside the path: enumerate the vertices of the path from $1$ to $n$. Suppose that $d_j(\theta) = 0$ for some $\theta \in (-2, 2)$. Then $d_{j-1}(\theta) = \infty$, and so, choosing $\epsilon$ sufficiently small, we have $\theta \pm \epsilon \in (-2, 2)$; and $d_{j-1}(\theta-\epsilon) > 0$, $d_{j-1}(\theta+\epsilon) < 0$, $d_{j}(\theta-\epsilon) < 0$, $d_{j}(\theta+\epsilon) > 0$. Thus the number of values $\geq 0$ within the path can only increase when the root becomes equal to $0$. But for $n=7$ we need to increase the number of values $\geq 0$ in the path $3$ times. Therefore, this counting alone is insufficient to forbid integrality. We further demonstrate that the root needs to become equal to $0$ once more, by analysing vertex $n$ this time. 
    
    Assume $d_{n}(\alpha) \leq 0$ for some $\alpha \in (0, 2]$. Since the function is odd, $d_{n}(-\alpha) \geq 0$ an $d_n$ passes through a pole in the interval $(-2, 2)$, implying that the number of positive values in the path has decreased and $d_1$ needs to become equal to $0$ once more.
    
    Now, assume $d_{n}(\alpha) \geq 0$ for all $\alpha \in (0, 2]$. In this case, $d_n(2) \geq 2$ since $d_n' \geq 1$. However, as the function $2-1/x$ maps $[1, \infty]$ to itself we actually have at least $n-1$ positive values or poles along the path when we make $\theta=2$. This implies that we need to increase the number of values $\geq 0$ in the path $5$ times instead of $3$ times.
    
    Therefore, either $d_n$ passed through a pole, or we already had enough eigenvalues in the first counting argument. In both cases, we have at least four eigenvalues of $T$ in the interval $(-2, 2)$, so at least one of then is not an integer.
\end{proof}

\section*{Acknowledgements}

The author acknowledge the financial support from CNPq and FAPEMIG.

\bibliographystyle{plain}
\IfFileExists{references.bib}
{\bibliography{references.bib}}
{\bibliography{../references}}

	
\end{document}